\documentclass[10pt,letterpaper]{amsart}
\usepackage{amsthm}
\usepackage{hyperref}
\usepackage{xcolor}

\title{Uniqueness of the maximal solution of the supercooled Stefan problem in 1D}
\author{Kai Hong Chau}
\address{Kai Hong Chau\\ Department of Mathematics \\ University of
  British Columbia\\ Vancouver, Canada}
\email{1415jameschau@math.ubc.ca}
\author{Young-Heon Kim}
\address{Young-Heon Kim\\ Department of Mathematics \\ University of
  British Columbia\\ Vancouver, Canada}
\email{yhkim@math.ubc.ca}
\author{Mathav Murugan}
\address{Mathav Murugan\\ Department of Mathematics \\ University of
  British Columbia\\ Vancouver, Canada}
\email{mathav@math.ubc.ca}

\keywords{} \subjclass[]{} \thanks{ 
MM is partially supported by
NSERC (RGPIN-2024-06125) and the Canada Research Chairs program (CRC-2023-00048).
YHK is
  partially supported by the Natural Sciences and Engineering Research
  Council of Canada (NSERC), with Discovery Grant RGPIN-2019-03926 and RGPIN-2025-06747, as
  well as Exploration Grant (NFRFE-2019-00944) from the New Frontiers
  in Research Fund (NFRF). YHK is also a member of the Kantorovich
  Initiative (KI), which is supported by the PIMS Research Network
  (PRN) program of the Pacific Institute for the Mathematical Sciences
  (PIMS). We thank PIMS for their generous support. Part of this work
  was completed during YHK's visit at the Korea Advanced Institute of
  Science and Technology (KAIST), and we thank them for their
  hospitality and the excellent environment. \copyright 2026 by the
      authors. All rights reserved.}

\date{\today}
\usepackage{amsmath,amsfonts,amssymb,amsthm,epsfig,epstopdf,url,array,graphicx,enumitem,quotes,hyperref}
\usepackage{scalerel,stackengine}
\usepackage{marginnote}
\usepackage{color}
\usepackage{ dsfont }
\usepackage{authblk}
\usepackage{fancyhdr}
\usepackage{tikz}
\stackMath
\newcommand\reallywidecheck[1]{
\savestack{\tmpbox}{\stretchto{
  \scaleto{
    \scalerel*[\widthof{\ensuremath{#1}}]{\kern-.6pt\bigwedge\kern-.6pt}
    {\rule[-\textheight/2]{1ex}{\textheight}}
  }{\textheight}
}{0.5ex}}
\stackon[1pt]{#1}{\scalebox{-1}{\tmpbox}}
}
\usepackage{fancyhdr} 
\newtheorem{theorem}{Theorem}[section]
\newtheorem{definition}[theorem]{Definition}
\newtheorem{remark}[theorem]{Remark}

\newtheorem{corollary}[theorem]{Corollary}
\newtheorem{example}[theorem]{Example}
\newtheorem{lemma}[theorem]{Lemma}

\begin{document}
\begin{abstract}
    We prove uniqueness of the maximal weak solutions to the supercooled Stefan problem in 1 dimension. This follows by showing that in 1 dimension, 
    the optimal solution of the corresponding free target optimal transport problem given in \cite{GeneralDimensions}, is independent of the choice of the cost function. Moreover, we show that the supercooled Stefan problem lacks monotonicity and $L^1$-Lipschitz stability, which are available in a similar problem considered in a previous paper \cite{freetarget}. However, in $1$ dimension, it has stability in the weak convergence of measures.
\end{abstract}

\maketitle

	\newcounter{num}
	\setcounter{num}{1}
    
\pagestyle{fancy}
\fancyhead{} 
\fancyhead[COH]{\textsc{Uniqueness of the maximal solution of the supercooled Stefan problem in 1D}}
\fancyhead[CEH]{\textsc{Kai Hong Chau, Young-Heon Kim and Mathav Murugan}}

\tableofcontents

\section{Introduction}
The supercooled Stefan problem is a free boundary problem for describing the dynamics of freezing of supercooled water. It can be  described as follows \cite{GeneralDimensions}:

Let $\eta$ be a function on $\mathbb{R}^d$ such that $-\eta$ describes the temperature of supercooled water \cite{Eisbildung,Eisbildung2}. Let $\chi_{\{\eta>0\}}$ be the unit amount of change in latent heat when water turns into ice. The supercooled Stefan problem is the following free boundary PDE:
 \begin{equation}\label{$S_t$}
 \tag{\textit{St}}
 \begin{cases}
     & (\eta-\chi_{\{\eta>0\}})_t-\dfrac{1}{2}\bigtriangleup \eta=0 \ \text{on } \mathbb{R}^d\times (0,\infty),\\ 
     &   \eta(\cdot,0)=\eta_0\in L^1(\mathbb{R}^d),\\  
& E=\lim \sup \limits_{t\to 0^+} \{\eta(\cdot,t)>0\}. 
 \end{cases}
 \end{equation}

The conditions of existence and non-existence of solution of $(\ref{$S_t$})$ have  been discussed in the literature; see e.g. \cite{GeneralDimensions,parabolic}. There is no uniqueness of the solution of $(\ref{$S_t$})$; see, e.g.  \cite[Proposition 6.3]{freetarget}. In the meantime, there is a lot of work on uniqueness of a specific type of solutions in $1$ dimension $d=1$:  For instance, on the domain $\{x\geq 0\}$, \cite{McKean-Vlasov} shows local uniqueness of McKean–Vlasov problem, by probabilistic reformation as in the approach in \cite{blow-ups} and under local assumptions on the initial data $\eta_0$. The papers \cite{kinetic,blow-ups,regularity} analyze nice local H\"older continuity or local continuous differentiability of the free boundary $\Lambda_t$ in time $t$,
and show  global existence of certain physical solutions. Moreover, \cite{1D-uniqueness} shows that given nice conditions on $\eta_0$, there is uniqueness of specific type of solutions $\eta$ of the two-sided Stefan problem.

In the present paper, we consider  uniqueness of maximal solutions of $(St)$ defined in (\ref{definition:maximal}). It should be noted that unlike the approach in \cite{kinetic,blow-ups,regularity}, we just analyze the subharmonically-ordered relationship between different primal solutions, and do not analyze any property regarding the evolution of  $\Lambda_t$.

Our approach is based on the results of \cite{GeneralDimensions}.
Assume that  $O := \{\eta_0 > 0\}$ be an open set. Then, \cite{GeneralDimensions} considered 
a stochastic optimization problem of the Brownian motion $W_t$ in the following form:
\begin{equation}\label{primal}
 P(\mu,u,O)=\inf\left\{ \left. \int_O u(x) dv(x)\right|\tau\leq \tau^O, W_{\tau}\sim v\leq \chi_{O} \text{ and } W_0\sim \mu \right\}      
\end{equation}
for each given strictly superharmonic function $u: \overline{O} \to \mathbb{R}$.
Here $\tau$ is a stopping time of the Brownian motion, and $\tau^O$ denotes the exit time with respect to the domain $O$.
Let $\eta_0\leq_{SH,O}\nu$ denotes subharmonic order (see Section~\ref{sec: notation}). In \cite[Theorem 6.5]{GeneralDimensions},  it is shown that if $\mu=\eta_0$ satisfies (\ref{($C_0$)}):
\begin{align*}\label{($C_0$)}
 \tag{$C_0$}
 &\text{There exists a constant } 0 <\delta < 1  \text{, such that } \eta_0\leq_{SH,O} \nu\\
 &\text{with respect to $O$ 
 for some $\nu$ with } \nu\leq \delta \chi_O,
\end{align*}

then optimal solution $\nu^*$ of $ P(\eta_0,u,O)$ exists and
 is a characteristic function \begin{equation}\label{primal shape}
\nu^*=\chi_{A} .    
\end{equation} 

Then, from  \cite[Theorem 7.4]{GeneralDimensions},  when $|\{\eta_0=1\}|=0$, 
there exists  a weak solution $\eta$ to $(\ref{$S_t$})$ with the initial data $\eta_0$ and the initial domain
$E = O$ (a.e.), which corresponds to the Eulerian variable associated to the optimal solution $\nu^*$ to $P(\eta_0,u,O)$. Moreover, $\eta$ is maximal in the sense of 
\cite[Definition 1.3]{GeneralDimensions}:
\begin{definition}[Maximal solution]\label{definition:maximal}
     A weak solution $\eta$ to $(\ref{$S_t$})$ is maximal if for any other weak solution $\tilde{\eta}$ with the same initial data $\eta_0$ and initial domain $\{\eta_0>0\}$,
     $$\chi_{\Sigma(\eta)}\leq_{SH} \chi_{\Sigma(\tilde{\eta})} \text{ implies } \eta=\tilde{\eta}.$$
Here,
the transition zone $\Sigma(\eta)$ is defined as the region where supercooled water turns into ice, namely, 
$$\Sigma(\eta)=\{x \ | \  \eta(x,t) = 0 \text{ for some } 0<t<\infty\}\bigcap \{\eta_0 > 0\}.$$

\end{definition}
Note that in fact, $\nu^* = \chi_{\Sigma(\eta)}$ for $\eta$ generated b $\nu^*$  \cite[Theorem 7.2]{GeneralDimensions}.

 For each fixed $u$, the optimizer $\nu^*$   of $P(\mu,u,O)$  is uniquely determined (\cite[Corollary 6.6]{GeneralDimensions}). We are interested in whether such optimal solution $\nu^*$ 
is independent of $u$, and whether the maximal solution of $(St)$ is unique. In 1 dimension, we show this affirmatively:
\begin{theorem}\label{thm:main}
    Let $O\subset \mathbb{R}$ be a bounded open set. Let $\mu=\eta_0$ be given and satisfy  (\ref{($C_0$)}). Then there exists  unique $\nu^*$ (described in $(\ref{shape})$ below), which is the optimal primal solution of $(\ref{primal})$ for all smooth superharnonic function $u$ (see $(\ref{($u$)})$).
\end{theorem}

Our proof in Section~\ref{sec:main proof} heavily relies on the fact that every open set in $1$D can be expressed as a countable disjoint union of open intervals. 
 Based on Theorem \ref{thm:main}, we have uniqueness of maximal solutions (in $1$D)  as a corollary:

\begin{theorem}\label{thm:maximal}
Given the setting of \cite[Theorem 7.4]{GeneralDimensions}, in $1$ dimension, given initial data $(\eta_0,O\stackrel{\text{a.e.}}{=} \{\eta_0>0\})$, the maximal solution $\eta$ of $(\ref{$S_t$})$ is unique, with $\chi_{\Sigma(\eta)}=\nu^*$ and $\nu^*$ is the unique optimal primal solution from Theorem \ref{thm:main}.
\end{theorem}
 One can also prove (see Section~\ref{sec:stability}) that the unique maximal solution in $1$D case, is stable in weak convergence,  while it does not enjoy a stronger $L^1$-Lipschitz stability from monotonicity property in the usual stable Stefan problem case \cite{freetarget}.

\begin{remark}
We remark that while we were preparing this article we found that \cite{unique-supercooled} obtained a similar result in 1D case along with other results in general dimensions, including dependence of $\nu^*$ on the choice of superharnonic function in dimension $\ge 2$.
\end{remark}

\section{Notation}\label{sec: notation}
The following  set up the notation used in the paper. Most of them are based on \cite{GeneralDimensions,freetarget}:
\begin{itemize}
    \item $O$ is a bounded open set in $\mathbb{R}^d$, with locally Lipschitz boundary.
  \item $u$ is a function satisfying the following condition:
  \begin{equation}\label{($u$)}
  u:\overline{O}\to \mathbb{R}, \text{ a bounded, smooth function on  $\overline{O}$ such that }    \bigtriangleup u<0.
  \end{equation}

   \item $\Omega=C(\mathbb{R}_{\geq 0},\mathbb{R}^d)$, the space of continuous curves.

\item $\mathbb{P}^{\zeta}=$ the Wiener measure on $\Omega$ with initial distribution $\zeta$; this means for our Brownian motion $W_t$, $W_0\sim \zeta$.
   
\item $(\Omega,\mathcal{F},\{\mathcal{F}_t\}_{t\geq 0},\mathbb{P}^{\zeta})$, with filtered probability space over $\Omega$, the filtered $\sigma$-algebra $\mathcal{F}$ and the probability measure $\mathbb{P}^{\zeta}$.

 \item  $\tau^A$ is the exit time of $A\subset \mathbb{R}^d$, i.e. In the setting $(\Omega,\mathcal{F},\{\mathcal{F}_t\}_{t\geq 0},\mathbb{P}^{\zeta})$, for all $\omega \in \Omega$ , $\tau^A(\omega)=\inf \{t>0|W_t(\omega)\in A^c\}$ .

  \item $M(\mathbb{R}^+)=$ collection of probability measures on $\mathbb{R}^+$.
   
     \item $\tau$ is a Randomized stopping time described as follows:  $\tau:\Omega\to M(\mathbb{R}^+)$ such that for all $t\geq 0$, 
     $\omega \to \tau(\omega)([0,t])$ is $\mathcal{F}_t-$measurable.

     \item $k=$ mass of a measure $\zeta$: $k=\int  d\zeta(x) =\zeta(X)$. 
    \item $\beta$: Mean/first moment of a measure $\zeta$: $\beta=\int  x \, d \zeta(x) $.

     \item $\leq_{SH}$ represents subharmonically ordered-relationship, which is described as follows: $\mu\leq_{SH}\nu$ if there is a (randomized) stopping time $\tau$ such that $W_0\sim \mu$, $W_{\tau}\sim v$, $\mathbb{E}(\tau)<\infty$. 

       \item $\leq_{SH,O}$: represents subharmonically ordered-relationship with respect to $U$, which is described as follows: $\mu\leq_{SH,O}\nu$ if there is a (randomized) stopping time $\tau$ such that $\tau\leq \tau^O$, $W_0\sim \mu$, $W_{\tau}\sim v$, $\mathbb{E}(\tau)<\infty$.

\end{itemize}

\section{Newtonian potentials}
To begin with, we prove a lemma regarding Newtonian potentials and subharmonic order. The results in this section hold in general dimensions.

 Newtonian potentials  on $\mathbb{R}^d$, are defined as 
$$k(y):=\begin{cases}
    -\dfrac{1}{2}|y| & \ \text{$d=1$,}\\
    -2\pi \log{|y|} & \ \text{$d=2$,}\\
    \dfrac{1}{d(d-2)\omega_d}|y|^{2-d}& \ \text{otherwise.}
\end{cases}$$
 where $\omega_d$ is the volume of the unit ball in $\mathbb{R}^d$. Note that this satisfies $$\Delta N(y)=-\delta_0.$$
For any finite Borel measure $\mu$, as in \cite{potentials}
one can  define the potential $U^{\mu}$ of $\mu$ by $$U^{\mu}(y)=\int k(y-x) \ d\mu(x).$$
 The function $U^{\mu}$ is a superharmonic on $\mathbb{R}^d$ as $$\Delta U^{\mu}(y)=-\mu(y)\leq 0 \ \text{for all $y\in \mathbb{R}^d$}.$$

Denote  
$$\mathcal{A}_{\mu,O}=\{\nu:\mu\leq_{SH,O} \nu \text{, } \nu \leq \chi_{O}\}$$
with a partial order relation $\leq_{SH,O}$. From \cite{GeneralDimensions,embedding,PDESkorokhod}, if $\mu \leq_{SH, O} \nu$, then there exists a stopping time $\tau^*\leq \tau^O$ with $\mathbb{E}(\tau^*)<\infty$ such that $W_0\sim \mu$ and $W_{\tau^*} \sim \nu$. 
So $\mathcal{A}_{\mu,O}$ is the admissible set of  $P(\mu,u,O)$,  that is,
$$\mathcal{A}_{\mu,O}=\{\nu:\nu\leq \chi_{O}\text{,  $\exists$ stopping time $\tau\leq \tau^O$ s.t.   } W_0\sim \mu,  W_{\tau}\sim \nu \}.$$

In the meantime, \cite{potentials} suggests a way to characterize subharmonic order by using Newtonian potentials.
\begin{lemma}\label{lemma:potentials} Suppose $O$ is a bounded, convex open set such that \begin{enumerate}
    \item $\mu,\nu$ have bounded densities, same mass and first moment, and are compactly supported on $\overline{O}$, and 
    \item with $W_0\sim \mu$, $\tau^{O}=\tau^{\overline{O}}$ almost surely.
    
    \end{enumerate} 
Then 
      $$ \mu \leq_{SH, O} \nu  \ \ 
      (\Longleftrightarrow) \ U^{\mu}(x)\geq U^{\nu}(x) \ \text{for all }x\in \mathbb{R}^d. $$
\end{lemma}

\begin{proof}
      $(\Longrightarrow))$ This comes from the fact that potential function is superharmonic and  \cite[Lemma 2.4]{GeneralDimensions}.

        $(\Longleftarrow)$ Firstly, as $\mu,\nu$ are compactly supported finite measures on $\overline{O}$ with bounded densities, from \cite{potentials}, $U^{\mu},U^{\nu}$ are finite and continuous on $\mathbb{R}^d$. Then we have to verify that  $\lim \limits_{|x|\to +\infty} [U^{\mu}(x)-U^{\nu}(x)]=0$.
           \begin{enumerate}[label=\textbf{Case  \arabic*.},start=1]
           \item $d\geq 3$: Trivial from \cite{potentials}.
\item $d=1$:
For $|x|\geq \sup \limits_{z\in \overline{O}}|z|$,
\begin{align*}
    U^{\mu}(x)-U^{\nu}(x)
    &=\begin{cases}
  -\frac{1}{2}\int_{\overline{O}}[x-y] [\mu(y)-\nu(y)] \ dy \ \text{when } x\geq 0,\\  
   -\frac{1}{2}\int_{\overline{O}}[y-x] [\mu(y)-\nu(y)] \ dy \ \text{when } x<0
    \end{cases}\\
    &=0.
\end{align*}
The last equality is true since $\mu,\nu$ have same mass and first moment.
\item $d=2$: For large $|x|\gg 1$, when $\zeta=\mu$ or $\nu$, 
        $$\int_{\overline{O}}-2\pi \log{(|x|+\sup \limits_{z\in \overline{O}}|z|) }\zeta(y)  dy\leq U^{\zeta}(x)\leq \int_{\overline{O}}-2\pi \log{(|x|-\sup \limits_{z\in \overline{O}}|z|) }\zeta(y)  dy$$
       
Hence with $\int_{\overline{O}} \mu(y) \ dy=\int_{\overline{O}} \nu(y) \ dy$, 
\begin{align*}
      &\limsup_{|x|\to \infty}\,  [U^{\nu}(x)-U^{\mu}(x)]\\
      &\leq  \limsup_{|x|\to \infty}-2\pi \left[ \log{(|x|-\sup \limits_{z\in \overline{O}}|z|) }-\log{(|x|+\sup \limits_{z\in \overline{O}}|z|) } \right]  \int_{\overline{O}} \mu(y) \ dy\\
        &= \limsup \limits_{|x|\to \infty}-2\pi \log\left[1-\dfrac{2\sup \limits_{z\in \overline{O}}|z|}{|x|+\sup \limits_{z\in \overline{O}}|z| } \right] \int_{\overline{O}} \mu(y) \ dy\\
        &= 0 .
 \end{align*}
 
As $\mu$ and $\nu$ in the argument above can be interchanged, $$\liminf_{|x|\to \infty} \, [U^{\nu}(x)-U^{\mu}(x)]\geq 0.$$

 Therefore, $$  \lim \limits_{|x|\to +\infty}[U^{\nu}(y)-U^{\mu}(y)]=0.$$
           \end{enumerate}
By this and \cite{potentials}, we have  $\mu\leq_{SH}\nu$,  that there is a stopping time $\tau$ with $W_0 \sim \mu$,  $W_\tau \sim \nu$ and $\mathbb{E}(\tau)<\infty$. We now take care of the condition $\tau \le \tau^O$.
        
Denote $\mu_t$ as the distribution of $W_{\min{\{\tau,t\}}}$. From \cite[Remark 4.2]{freetarget}, for all $0\leq t< +\infty$, $supp(\mu_t)$ is contained in the convex hull of $supp(\nu)$. Since $\nu$ is compactly supported on $\overline{O}$ from our assumption $(2)$, and $\overline{O}$ is convex from the convexity of $O$, $supp(\mu_t)\subset \overline{O}$ , implying $\min\{\tau,t\}\leq \tau^{\overline{O}}$ for all $0\leq t<\infty$, which results in  $\tau\leq \tau^{\overline{O}}$ almost surely. By our assumption $(1)$ ,$\tau^{\overline{O}}=\tau^{O}$ almost surely, hence $\tau\leq \tau^{O}$ almost surely. To conclude, $\mu \leq_{SH,O}\nu$.
\end{proof}

\begin{remark} Note that $\tau^{\overline{O}}=\tau^O$ almost surely  for nice open domains $O$, e.g. bounded connected open sets with smooth boundary.

In general, the assumptions $\tau^{\overline{O}}=\tau^O$ and $\bar{O}$ is convex are necessary. For example, take $O=(-1,0)\bigcup (0,1)$, $\mu=\chi_{A}$ and $\nu=\chi_B$, where $A=(-0.5,0)\cup (0,0.5)$ and $B=(-1,-0.5)\cup (0,5,1)$. It is true that $W_{\tau^{O}}$ and $ W_{\tau^{\overline{O}}}$ do not have the same distribution. Upon calculation, one can show that $U^{\mu}\geq U^{\nu}$ on $\mathbb{R}$. Nevertheless, considering the conservation of mass and first moment  on $O=(-1,0)\bigcup (0,1)$, there is no stopping time $\tau\leq \tau^O$ with $\mathbb{E}(\tau)<\infty$ such that $W_0\sim \mu$ and $W_{\tau}\sim \nu$. The boundary point $0$ prohibits the flow of particles between $(-1,0)$ and $(0,1)$. Besides, $O$ is not convex.
\end{remark}

From Lemma~\ref{lemma:potentials}, we also derive that under the same assumptions (1) and (2) of that lemma, we have
$\mu =\nu$ if and only if $U^\mu = U^\nu$, since $\mu \leq_{SH,O}  \nu \leq_{SH,O}\mu$ implies $\mu =\nu$.
We also have the following corollary:

\begin{corollary}\label{cor:unique}
Given the conditions on $O$ in Lemma~\ref{lemma:potentials}, the following are equivalent:
  \begin{enumerate}[label=(\arabic*),start=1]
   \item $(\mathcal{A}_{\mu,O},\leq_{SH,O})$ has a unique maximal element in the partial ordering $\leq_{SH,O}$.
      \item The optimal solution $\nu^*$ for $P(\mu,u,O)$ is independent of $u$ satisfying (\ref{($u$)}).
      \item For all $\nu\in \mathcal{A}_{\mu,O}$, $\ U^{\nu}(x)\geq U^{\nu^*}(x) \ \text{for all }x\in \mathbb{R}^d$.
  \end{enumerate}
\end{corollary}
\begin{proof}
Lemma~\ref{lemma:potentials} shows that $(1)$ and $(3)$ are equivalent. It remains to show that $(1)$ and $(2)$ are equivalent.

    $(1)\Longrightarrow (2)$: Suppose $(\mathcal{A}_{\mu,O},\leq_{SH,O})$ has a unique maximal element, namely, $\nu^*\in \mathcal{A}_{\mu,O}$ such that for all $\nu\in \mathcal{A}_{\mu,O}$, we have $\nu\leq_{SH,O}\nu^*$. By the definition of subharmonic order and superharmonicity of $u$,
    \cite[Lemma 2.4]{GeneralDimensions} gives $\int_Ou(x) \ d\nu(x)\geq \int_Ou(x) \ d\nu^*(x)$. Hence 
$P(\mu,u,O)\geq \int_Ou(x) \ d\nu^*(x)$.
Moreover, as $\nu^*\in \mathcal{A}_{\mu,O} $, $P(\mu,u,O)\leq \int_Ou(x) \ d\nu^*(x)$.
Hence $P(\mu,u,O)= \int_Ou(x) \ d\nu^*(x)$ for all $u$ satisfying (\ref{($u$)}), and $\nu^*$ is  the optimal  solution for $P(\mu,u,O)$.

$(2)\Longrightarrow (1)$:
Suppose the optimal solution $\nu^*$ for $P(\mu,u,O)$ is independent of $u$. Then $P(\mu,u,O)=\int_Ou(x)\nu^*(x) \ dx$ for all $u$ satisfying (\ref{($u$)}). Let $\nu_1\in \mathcal{A}_{\mu,O}$, $y\in \mathbb{R}^d$. Note that for $k_y(x)=k(y-x)$, as $k_y\in L^1_{loc}(\mathbb{R}^d)$, there exists a sequence of mollifier $\{\eta_{\epsilon}\}_{\epsilon>0}$ such that $k_y\ast \eta_{\epsilon}\stackrel{L^1({\overline{O}})}{\to} k_y$ and $|k_y|\ast \eta_{\epsilon}\stackrel{L^1({\overline{O}})}{\to} |k_y|$.
Moreover,  $k_y$ is superharmonic on $O$ and $k_y\ast \eta_{\epsilon}$ is smooth superharmonic function on $O$. For all $C>0$, there exists $D>0$ such that  $u_{y,\epsilon,C,D}=k_y\ast \eta_{\epsilon}-C|x|^2+D$ is a non-negative bounded, smooth function with $\Delta y_{y,\epsilon,C,D}<0$ on $O$. By definition of optimal solutions of $P(\mu,u,O)$,
we have
$\int_O u_{y,\epsilon,C,D} (x) \ \nu^*(x) \ dx \leq  \int_O u_{y,\epsilon,C,D} (x) \nu_1(x) \ dx$. With $\nu^*(\overline{O})=\nu_1(\overline{O})$, we obtain 

\begin{align*}
\int_{\overline{O}} [k_y\ast \eta_{\epsilon}(x)-C|x|^2] \nu^*(x) \ dx&\leq  \int_{\overline{O}} [k_y\ast \eta_{\epsilon}(x)-C|x|^2] \nu_1(x) \ dx. \\
\text{By taking $C\downarrow 0$, } \int_{\overline{O}} k_y\ast \eta_{\epsilon}(x) \nu^*(x) \ dx&\leq \int_{\overline{O}} k_y\ast \eta_{\epsilon}(x)\nu_1(x) \ dx.
\end{align*}
Moreover, as $\left|k_y\ast \eta_{\epsilon}(x) \nu^*(x)\right|\leq|k_y|\ast \eta_{\epsilon}$ and $|k_y|\ast \eta_{\epsilon}\stackrel{L^1({\overline{O}})}{\to} |k_y|$, by generalized dominated convergence theorem,  $$\lim \limits_{\epsilon\to 0}\int_{\overline{O}} k_y\ast \eta_{\epsilon}(x) \nu^*(x) \ dx=\int_{\overline{O}}k_y (x) \nu^*(x) \ dx= \int_{\overline{O}} k(y-x) \nu^*(x) \ dx=U^{\nu^*}(y).$$

Similarly, $$\lim \limits_{\epsilon\to 0}\int_{\overline{O}}k_y\ast \eta_{\epsilon}(x) \nu_1(x) \ dx=\int_{\overline{O}}k_y (x) \nu_1(x) \ dx=\int_{\overline{O}} k(y-x)\nu_1(x) \ dx =U^{\nu_1}(y).$$

Hence, $ U^{\nu^*}(y)\leq U^{\nu_1}(y)$ for all $y\in \mathbb{R}^d$, thereby  $\nu_{1}\leq_{SH,O}\nu^*$ by Lemma~\ref{lemma:potentials} and $\nu^*$ is the unique maximal element of $\mathcal{A}_{\mu,O}$.
\end{proof}

\section{Uniqueness In 1 dimension}\label{sec:main proof}
 In this section, we prove Theorems~\ref{thm:main} and \ref{thm:maximal}.

 We first explain the idea:
In 1D, any open set can be decomposed into  $$ O=\bigcup \limits_{n=1}^{\infty} (c_n,d_n),$$ where $(c_n,d_n)$ are mutually disjoint.  Let  $\mu=\sum \limits_{n=1}^{\infty} \mu_n$ with $supp(\mu_n)\subset [c_n,d_n]$  be given.
 Note that the constraint $\tau \le \tau^O$ for the optimization $P(\mu,u,O)$ does not allow the Brownian motion moves between different intervals, therefore the optimal solution to $P(\mu,u,O)$ itself decomposes the sum of optimal solutions $P(\mu_n,u_n,(c_n,d_n))$, where $\mu_n, u_n$ are the restrictions of $\mu$ and $u$ on $(c_n, d_n)$. 

We will prove that  
our optimal solution to $P(\mu,u,O)$ is given as
\begin{multline}\label{shape}
  \nu^*=\sum \limits_{n=1}^{\infty} \nu_n^*,\hbox{ \  where \ } \nu^*_n=\chi_{A_n} \hbox{ \ and \ }A_n=\left(c_n,e_n\right)\bigcup \left(f_n, d_n\right)\subset (c_n,d_n).\\
  \hbox{ \ such that  for each $n$, $\mu_n,\nu_n^* $  have same mass and first moment. \ } 
\end{multline}
 Here the point is that the same mass and the first moment condition and the specific form $A_n$ completely determines $\nu_n$ from $\mu_n$, thus the expression \eqref{shape} determines $\nu^*$ uniquely from $\mu$. 
In particular, when $O=(-1,1)$,
\begin{align}\label{eqn:A}
A=\left(-1,-1+\frac{k}{2}-\frac{\beta}{2-k}\right)\bigcup \left(1-\frac{k}{2}-\frac{\beta}{2-k}, 1\right).
\end{align}
The shape of $A$ comes from the following  reason: In \cite{GeneralDimensions}, it is shown that to minimize $\int u(x) \nu(x) \ dx$, we have to maximize $\tau$ as much as possible. With $\nu^*=\chi_{A}$, for each $n$, the particles should be saturated in intervals around the boundary points $c_n$ and $d_n$. The other endpoints of $\nu^*_n$ come from the property that the Brownian motion with $\mathbb{E}(\tau)<\infty$ preserves mass $k$ and first moment $\beta$. 
 Given $k$ and $\beta$, such endpoints are uniquely determined.

 We will first show the following theorem:
\begin{theorem}\label{thm:potentials}
   Let $O\subset \mathbb{R}$ be a bounded open set $O=\bigcup \limits_{n=1}^{\infty} (c_n,d_n)$.
   Let $\mu$ be an initial distribution of the form $\mu=\chi_E$ for a measurable set $E$ 
   and $\nu^*$ is given by $(\ref{shape})$. 
   Then we have  $\nu^*\in \mathcal{A}_{\mu,O}$ .
\end{theorem}
After proving this we will see in the next subsection that the optimal solution of \eqref{primal} is given by \eqref{shape}.

\begin{proof}[
Proof of Theorem~\ref{thm:potentials}]  
We proceed in the following steps:
\begin{enumerate}[label=\textbf{Step  \arabic*.},start=1]

\item  For $O=(-1,1)$, show that for $\nu^*=\chi_A$ with $A$ given in \eqref{eqn:A} it satisfies $\nu^*\in \mathcal{A}_{\mu,O}$ when $\mu=\chi_{(a,b)}$, where $-1<a<b<1$.\\
For general $O=(c,d)$ and $\mu=\chi_{(a,b)}$ such that $c<a<b<d$, by shifting and rescaling the domain, similar proof can be applied.

\item For $O=(-1,1)$, show that for $\nu^*=\chi_A$ with $A$ given in \eqref{eqn:A} it satisfies $\nu^*\in \mathcal{A}_{\mu,O}$ when $\mu=\sum \limits_{k=1}^n \chi_{(a_i,b_i)}$, where $(a_i,b_i)\subset(-1,1)$,  $\min \limits_{k} a_k>-1$ and $\max \limits_k b_k<1$.
\item For $O=(-1,1)$, show that   for $\nu^*=\chi_A$ with $A$ given in \eqref{eqn:A} it satisfies  $\nu^*\in \mathcal{A}_{\mu,O}$ when  $\mu=\chi_{E}$, where $E$  is a bounded measurable set.
\item For bounded open set $O\subset \mathbb{R}$, show that for $\nu^*$  given in \eqref{shape} it satisfies $\nu^*\in \mathcal{A}_{\mu,O}$ when $\mu=\chi_{E}$, where $E$ is a bounded measurable set.
\end{enumerate}

We will prove each of these steps below.
    \begin{enumerate}[label=\textbf{Step  \arabic*.},start=1]

\item  For $O=(-1,1)$, show that for $\nu^*=\chi_A$ with $A$ given in \eqref{eqn:A} it satisfies  $\nu^*\in \mathcal{A}_{\mu,O}$ when $\mu=\chi_{(a,b)}$, where $-1<a<b<1$.
\begin{proof}
 From Lemma~\ref{lemma:potentials}, it suffices to show  that $U^{\nu^*}\leq U^{\mu}$. As $\mu$ and $\nu^*$ have the same mass and the first moment, from Case 2 of the proof of 
Lemma~\ref{lemma:potentials},
$U^{\nu^*}(y)=U^{\mu}(y)$ for all $|y|\geq 1$. 
We now prove that  $U^{\nu^*}(y)\leq U^{\mu}(y)$ for all $|y|\leq 1$.

     For $\mu=\chi_{(a,b)}$ having mass $k$ and first moment $\beta$, by solving simultaneous equations,  we have $a=\dfrac{\beta}{k}-\dfrac{k}{2}$ and $b=\dfrac{\beta}{k}+\dfrac{k}{2}$.  To compare $U^{\nu^*}$ and $U^\mu$,
     note from direct computation, that  
\begin{align*}
    [U^{\nu^*}(y)]'=\begin{cases}
        -\dfrac{k}{2} &\text{for $y\geq 1$},\\
        -y+1-\dfrac{k}{2}&\text{for $1-\dfrac{k}{2}-\dfrac{\beta}{2-k}\leq y\leq1$},  \\ 
        \dfrac{\beta}{2-k}  &\text{for $-1+\dfrac{k}{2}-\dfrac{\beta}{2-k}\leq y\leq 1-\dfrac{k}{2}-\dfrac{\beta}{2-k}$},\\
     -y-1+\dfrac{k}{2}  &\text{for $-1 \leq y\leq -1+\dfrac{k}{2}-\dfrac{\beta}{2-k}$}, \\ 
       \dfrac{k}{2} &\text{for $y\leq -1$},   
    \end{cases}
\end{align*}
and 
\begin{align*}
    [U^{\mu}(y)]'&=\begin{cases}
     -\dfrac{k}{2} &\text{for } y\geq \dfrac{\beta}{k}+\dfrac{k}{2},\\
     -y+\dfrac{\beta }{k} &\text{for }  \dfrac{\beta}{k}-\dfrac{k}{2} \leq y \leq \dfrac{\beta}{k}+\dfrac{k}{2},\\
      \dfrac{k}{2} &\text{for } y\leq \dfrac{\beta}{k}-\dfrac{k}{2}.
    \end{cases}
\end{align*}

The following is to show $U^{\nu^*}(y)\leq U^{\mu}(y)$ for all $|y|\leq 1$:
\begin{proof} 
Notice $U^{\nu^*}-U^{\mu}\in C^1([-1,1])$.  Also, with $U^{\nu^*}(-1)-U^{\mu}(-1)=U^{\nu^*}(1)-U^{\mu}(1)=0$, and existence of $d=\min\{-1+\frac{k}{2}-\frac{\beta}{2-k},\frac{\beta}{k}-\frac{k}{2}\}>-1$ such that  $[U^{\nu^*}]'(s)-[U^{\mu}]'(s)<0$ on $(-1, d)$, we have $$\arg\min \limits_{y\in [-1,1]}\{ U^{\nu^*}(y)-U^{\mu}(y)\}\subset (-1,1) \text{ and } \min \limits_{y\in [-1,1]}\{ U^{\nu^*}(y)-U^{\mu}(y)\}<0.$$
So, a minimum point is a critical point $s_0 \in (-1,1)$ that $[U^{\nu^*}]'(s_0)-[U^{\mu}]'(s_0)=0$.
By computations, we can verify that there is unique critical point on $(-1,1)$; see the Appendix~\ref{sec:appendix} for details. 

Therefore, $$s_0=\arg\min \limits_{y\in [-1,1]}\{ U^{\nu^*}(y)-U^{\mu}(y)\},$$ and $$\max \limits_{y\in [-1,1]}\{ U^{\nu^*}(y)-U^{\mu}(y)\}=U^{\nu^*}(1)-U^{\mu}(1)=0.$$ To conclude,  $U^{\nu^*}(y)\leq U^{\mu}(y)$ for all $|y|\leq 1$. 
\end{proof}
\end{proof}

\item   For $O=(-1,1)$, show that  for $\nu^*=\chi_A$ with $A$ given in \eqref{eqn:A} it satisfies $\nu^*\in \mathcal{A}_{\mu,O}$ when $\mu=\sum \limits_{k=1}^n \chi_{(a_i,b_i)}$, where $(a_i,b_i)\subset(-1,1)$,  $\min \limits_{k} a_k>-1$ and $\max \limits_k b_k<1$.

\begin{proof}
 First consider the case where $\mu=\mu_1+\mu_2$, where $\mu_1= \chi_{(a,b)}$, $\mu_2= \chi_{(c,d)}$ and $-1<a<b<c<d<1$, and the remaining cases can be proved by induction.

 Consider $\nu_1^*=\chi_{(-1,a')\cup (b',c)}$ 
obtained from Step 1 with the interval $O=(-1, c)$ and the initial measure $\mu_1$. Then from  Step 1 applied to $\mu_1$, $\nu_1^*$, there exists a stopping time $\tau_1\leq \tau^{(-1,c)}\leq \tau^{(-1,1)}$ with $\mathbb{E}(\tau_1)<\infty$ such that $W_0\sim \mu_1$ and $W_{\tau_1}\sim \nu^*_1$. Now consider $I'=(b',d)$,   and let $\mu_2'= \chi_{I'}$.
Apply Step 1 to the underlying domain $O=(a',1)$, and get the resulting  $\nu^*_2$ in the form described in \eqref{shape}; see Figure 1. Then from Step 1, there exists a stopping time $\tau_2\leq \tau^{(a',1)}\leq \tau^{(-1,1)}$ with $\mathbb{E}(\tau_2)<\infty$ such that $W_0\sim \mu_2'$ and $W_{\tau}\sim \nu^*_2$.

Note that $\nu^*=\chi_{(-1,a')} +\nu^*_2$ is of the form $(\ref{shape})$ as $\nu^*=\chi_{(-1, a'')} $  and $\nu^*$, $\mu$ have the  same mass and the first moment. Let $\tau$ be the stopping time, $\tau=\tau_1\oplus \tau_2\leq \tau^{(-1,1)}$, where the direct sum $\oplus$ is understood as follows:
First apply the stopping time $\tau_1$ to the initial distribution $\mu_1$ and the zero stopping time $0$ to the other distribution $\mu_2$. Then, apply zero stopping time to the distribution $\chi_{(-1,a')}$ and $\tau_2$ to $\mu_2'$. Notice that for $\tau$,  $W_0\sim \mu$ and $W_{\tau}\sim \nu^*$ and  $\mathbb{E}(\tau)<\infty$. Hence $\nu^*\in \mathcal{A}_{\mu,O}$.

 Now for the general case 
$\mu=\sum \limits_{k=1}^n \chi_{(a_i,b_i)}$, we can apply the above argument inductively, starting from the three most left sub-intervals $(a_1,b_1)$, $(a_2, b_2)$, $(a_3, b_3)$, applying the procedure in the previous  paragraphs for 
$O=(-1, a_3)$ and 
$\mu_1= \chi_{(a_1, b_2)}$, 
$\mu_2=\chi_{(a_2, b_2)}$. After this procedure we get the new distribution of the form:
\begin{align*}
\mu'= \chi_{(-1, a_2')}+ \chi_{(a_3', b_3)} + \chi_{(a_4, b_4)} + \cdots.    
\end{align*} 
We then apply the previous procedure again and keep continuing until to get the final measure of the form
\begin{align*}
    \chi_{(-1, \bar a ) } + \chi_{(\bar b, 1)}.
\end{align*}
In this procedure, the total mass and the first moment does not change and this final form is given by \eqref{shape}.

\begin{figure}[h]\label{fig1}
\caption{1D diffusion of 2 intervals}
\centering
\includegraphics[scale=0.3]{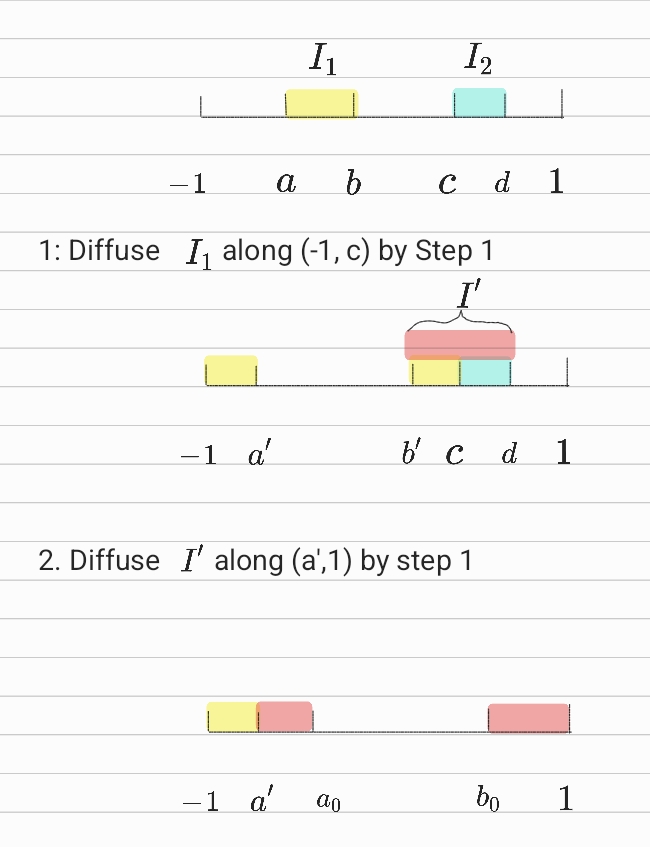}
\end{figure}
\end{proof}

\item  For $O=(-1,1)$, show that  for $\nu^*=\chi_A$ with $A$ given in \eqref{eqn:A} it satisfies $\nu^*\in \mathcal{A}_{\mu,O}$ when  $\mu=\chi_{E
}$,  where $E$ is a bounded measurable set.
\begin{proof}
 By outer regularity of the Lebesgue measure $m$ \cite[Theorem 1.18]{analysis} and the fact that every bounded open set in 1D is countable disjoint union of bounded open intervals, there is a sequence of sets $\{\bigcup \limits_{k=1}^{N_n} (c_{n_k},d_{n_k})\}_{n=1}^{\infty}$ such that 
 \begin{itemize}
     \item  $E_n=\bigcup \limits_{k=1}^{N_n} (c_{n_k},d_{n_k})$ is a finite union of disjoint open intervals and $-1<c_{n_k}<d_{n_k}<c_{n_{k+1}}<1$ for all $n,n_k$, and
     \item $\lim \limits_{n\to \infty} m(E_n)= m(E)$.
 \end{itemize}
Moreover, as $\mu_n=\chi_{E_n}$ converges to
$\mu=\chi_E$ in $L^1((-1,1))$, and the function $1,x,-\dfrac{1}{2}|x-y|$ are bounded on $(-1,1)$ for all $y\in \mathbb{R}$, we see by the dominated convergence theorem, convergence of the mass, the first moment and the potential of $\mu_n$ can be obtained, that is,
$\lim \limits_{n\to \infty}k_n=k$, $\lim \limits_{n\to \infty}\beta_n=\beta$, $\lim \limits_{n\to \infty} U^{\mu_n}(y)=U^{\mu}(y)$. Also, for $$\nu_n^* = \chi_{(-1, \bar a_n)} + \chi_{(\bar b_n, 1)}$$ given from Step 2 with $O=(-1,1)$ and the initial measure  $\mu_n$, we see that $\nu_n^*$ converges to $\nu^*$ in $L^1((-1,1))$. This is because those measures $\nu_n^*$ and $\nu^*$ are determined by $k_n, \beta_n$, $k, \beta$, respectively, in the form \eqref{shape}. Therefore, 
 $\lim \limits_{n\to \infty} U^{\nu_n}(y)=U^{\nu}(y)$ for all $y\in \mathbb{R}$. From Step 2 and Lemma~\ref{lemma:potentials}, $U^{\nu^*_n}\leq U^{\mu_n}$ on $\mathbb{R}$. Consequently, 
\begin{align*}
   U^{\nu^*}(y)=\lim \limits_{n\to \infty} U^{\nu^*_n}(y) \leq   \lim \limits_{n\to \infty} U^{\mu_n}(y)= U^{\mu}(y),
\end{align*}   
and $\nu^*\in \mathcal{A}_{\mu,O}$.
\end{proof}

\item For a bounded open set $O\subset \mathbb{R}$, show that   for $\nu^*$  given in \eqref{shape} it satisfies $\nu^*\in \mathcal{A}_{\mu,O}$ when $\mu=\chi_{E}$, where $E$ is a bounded measurable set.
\begin{proof}
 Note that $O=\bigcup \limits_{i=1}^{\infty} (a_i,b_i)$, for some mutually disjoint open sets $\{(a_i,b_i)\}_{i=1}^{\infty}$.

 For $\mu=\sum \limits_{i=1}^{\infty}\mu_i$
  with $\mu_i = \mu\big|_{(a_i, b_i)}$ we apply Step 3, with $O=(a_i, b_i)$, then for each $\mu_i$ and the corresponding $\nu_i^*$ described in (\ref{shape}), by Step 3, there exists a stopping time $\tau_i\leq \tau^{(a_i,b_i)}\leq\tau^O$ with $\mathbb{E}(\tau_i)<\infty$ such that $W_0\sim \mu_i$ and $W_{\tau_i}\sim \nu^*_i$.
  
 Take $\tau=\bigoplus \limits_{i=1}^{\infty} \tau_i$  where this $\oplus$ is understood as that $\tau=\tau_i$ when it is applied to each part $\mu_i$ of the initial distribution $\mu$.  Note that $\tau\leq \tau^O$ (hence $\mathbb{E}(\tau)<\infty$) and  since $(a_i,b_i)$'s are mutually disjoint, we have $W_0\sim \mu$, $W_{\tau}\sim \sum \limits_{i=1}^{\infty} \nu^*_i=\nu^*$. Therefore $\nu^*\in \mathcal{A}_{\mu,O}$.
\end{proof}
\end{enumerate}
\end{proof}

\subsection{Proof of Theorems~\ref{thm:main} and ~\ref{thm:maximal}}\label{sec:proof}
 Given $\mu$ satisfying  (\ref{($C_0$)}), 
notice that as mentioned in $(\ref{primal shape})$, any optimal solution $\gamma$ to $P(\mu,u,O)$  
  is of the form $\gamma=\chi_{E}$. 
  Consider $\nu^*$ given in Theorem~\ref{thm:potentials} (taking $\mu=\gamma$ there); thus $\gamma \leq_{SH,O} \nu^*$; since $\mu\le_{SH,O}\gamma$, this also implies $\nu^* \in\mathcal{A}_{\mu,O}$. 

  Since $P(\mu,u,O)=\int u d\gamma \ge \int u  d\nu^*$, and for fixed $u$, optimizer $\gamma$ of $P(\mu,u,O)$ is unique (\cite[Corollary 6.6]{GeneralDimensions}), this means that $\nu^* =\gamma$. This proves that $(2)$ of Corollary~\ref{cor:unique} holds. Therefore, Theorem~\ref{thm:main} holds. Also, $(\mathcal{A}_{\mu,O},\leq_{SH,O})$ has a unique maximal element $\nu^*$. Our proof shows that the optimal solution is of the form \eqref{shape} since it is from Theorem~\ref{thm:potentials}.

\begin{proof}[Proof of Theorem~\ref{thm:maximal}] 
    Let $\eta$ be a maximal solution of (\ref{$S_t$}) with initial data $(\mu=\eta_0,O\stackrel{\text{a.e.}}{=} \{\eta_0>0\})$,  $\nu=\chi_{\Sigma(\eta)}$. Then $\nu\in \mathcal{A}_{\mu,O}$ due to \cite[Theorem 7.2]{GeneralDimensions}.  Let $\eta^*$ be the (weak) solution of (\ref{$S_t$})
    generated by $\nu^*$, as given in \cite[Theorem 7.4]{GeneralDimensions}, with the same initial condition;  it holds $\nu^*=\chi_{\Sigma(\eta^*)}$. Due to Theorem~\ref{thm:potentials},  $\chi_{\Sigma(\eta)}=\nu\leq_{SH} \nu^*=\chi_{\Sigma(\eta^*)}$, therefore $\eta=\eta^*$ by maximality of $\eta$. As a result,  $\eta^*$ is the unique maximal weak solution of (\ref{$S_t$}).
\end{proof}

\section{Stability}\label{sec:stability}
In \cite[Section 7]{freetarget}, it is shown that for initial distributions $\mu_1\leq \mu_2$, if $\nu_1,\nu_2$ both solve the optimal free target problem in their setting, then $\nu_1\leq \nu_2$. This monotonicity then yields $L^1$ contraction principle: $\|(\nu_1-\nu_2)_+\|_{L^1}\leq \|(\mu_1-\mu_2)_+\|_{L^1}$.
 Such monotonicity is from the feature that in their setting  \cite{freetarget} the optimizer $\nu$  minimizes the corresponding stopping time $\tau$ as much as possible.
On the contrary, due to superharmonicity of $u$, the solutions to  \eqref{primal} maximize the stopping times. Indeed, we see below that the monotonicity and $L^1$-contraction property in  \cite[Section 7]{freetarget}  do not hold in our setting.
\subsection{Lack of monotonicity}

 We give the counter-examples in 1D, but, one can easily extend them as a $d$-dimensional result.
\begin{example}
   An easy  example is $\mu_1=\chi_{(-0.9,0)}$ and $\mu_2=\chi_{(-1,0)}$ on $O=(-1,1)$.  Here   $\mu_1\leq \mu_2$, but, $\nu_1\not \leq \nu_2$. 
  To see this, notice that for $\mu_2$ as the initial distribution, all Brownian particles confined in $O=(-1,1)$ are already saturating the density upper bound constraint  around the boundary point $-1$. Around this point, to satisfy $\tau\leq \tau^{(-1,1)}$, the particles cannot diffuse. Then, at any point $y> -1$ close $-1$, the mass cannot diffuse either from $\mu_2$, because, otherwise, it will increase the mass on the left of $y$, which is already saturated  by the density upper-bound constraint. Then one can apply this argument in a continuously inductive way with respect to $y$, to see that no mass can diffuse from $\mu_2$. As a result $\nu_2=\mu_2$.
  However, for $\mu_1$, the particles can diffuse both left and rightwards, and $\nu_1$ will have positive mass on $(0,1)$. Therefore $\nu_1 \not \le \nu_2$.  
\end{example}
     
 For a more nontrivial counter-example, we consider the case $\mu_1,\mu_2$ have the same first moment $\beta$. 

\begin{example}
     Take $\mu_1=0.99\chi_{(0,\sqrt{0.75})}$, $\mu_2=0.99\chi_{(-0.5,1)}$ and $O=(-1,1)$. Then $\mu_1,\mu_2$ satisfies ,
$\mu_1\le \mu_2$, and $\beta_1=\beta_2=0.37125$.  Then, as the optimal solutions $\nu_i$ of \eqref{primal} are given by \eqref{shape} as seen in the proof of Theorem~\ref{thm:main}, one can determine them as 
    \begin{itemize}
     \item $\nu_1=\chi_{A_1}$, where $A_1=\left(-1,-0.896224371 \right) \bigcup \left(0.246410478 ,1 \right)$.
     
        \item $\nu_2=\chi_{A_2}$, where $A_2=\left(-1,-0.978373786 \right) \bigcup \left(-0.463373786 ,1 \right)$.
       
    \end{itemize}
     Here $\mu_1\leq \mu_2$  but $\nu_1\not \leq \nu_2$.

\end{example}

\subsection{Lack of $L^1$-Lipschitz Stability}
 We show that there is no $L^1-$Lipschitz relationship between $\mu$ and $\nu$ in 1D.

\begin{lemma}
There is no $C>0$ such that for every $\nu_i$, the optimizer of \eqref{primal} with the initial distribution $\mu_i$, $i=1,2$, $\|(\nu_1-\nu_2)_+\|_{L^1}\leq C\|(\mu_1-\mu_2)_+\|_{L^1}.$
\end{lemma}

\begin{proof}
 As we have seen from the proof of  Theorem~\ref{thm:main}
in Section~\ref{sec:proof}, $\nu_i$'s  are described in $(\ref{shape})$. Let $0<x<1$, $0<y\ll 1$, $0<r< 1$, $c\approx 1$ such that $-c+ry<-x$, and take $\mu_1=r\chi_{(-x,x)}$ and $\mu_2=\chi_{(-c,-c+ry)}+r\chi_{(-x,x-y)}$. Then from  \eqref{shape}, one can calculate
$\nu_i=\chi_{A_i}$, where
\begin{itemize}
    \item $A_1=(-1,-1+rx)\bigcup (1-rx,1)$ and 
    \item $A_2=(-1,-1+rx+\frac{2rxy+2cry-ry^2-r^2y^2}{2(2-2rx)} )\cup (1-rx+\frac{2rxy+2cry-ry^2-r^2y^2}{2(2-2rx)},1).$
\end{itemize}

When $y\approx  0$, $c\approx 1$, we will have the following picture:
\begin{figure}[h]
\caption{Shape of $\mu_1$, $\mu_2$}
\centering
\includegraphics[scale=0.21]{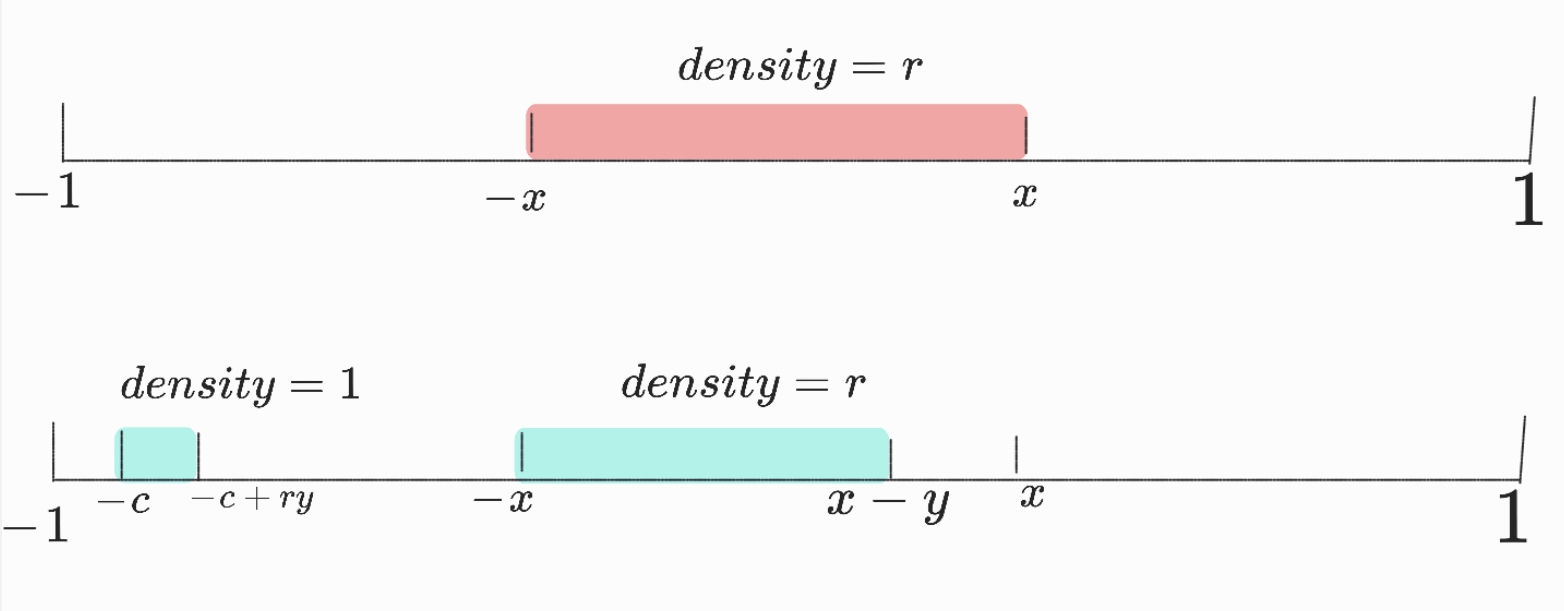}
\end{figure}
\begin{figure}[h]
\caption{Shape of $\nu_1$, $\nu_2$}
\centering
\includegraphics[scale=0.28]{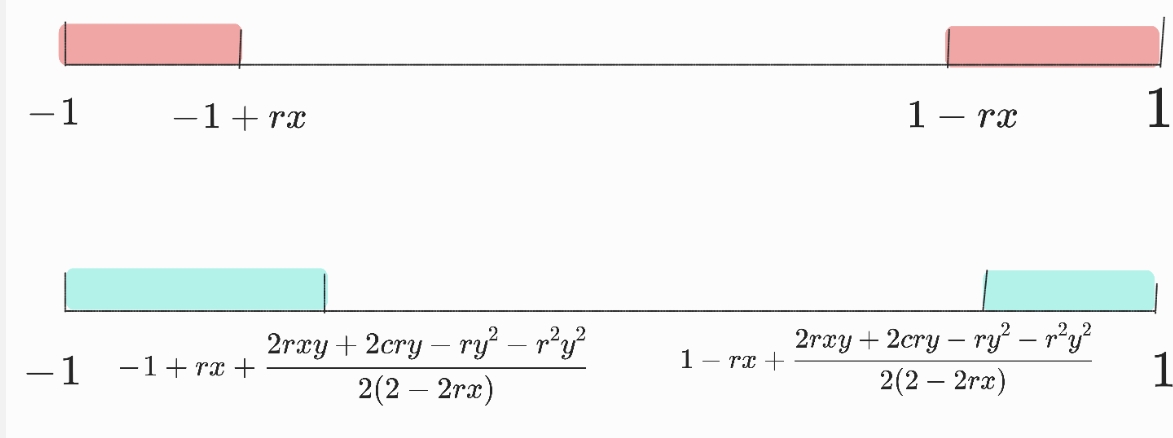}
\end{figure}

Then $\int|(\mu_1-\mu_2)_+|=ry$ and $\int|(\nu_1-\nu_2)_+|=\dfrac{2rxy+2cry-ry^2-r^2y^2}{2(2-2rx)}$. Suppose $C$ is a Lipschitz constant, we have 
$$C\geq \dfrac{2x+2c-y-ry}{4(1-rx)}\stackrel{y\downarrow 0}{\longrightarrow}   \dfrac{x+c}{2(1-rx)}\stackrel{c,x,r\uparrow 1}{\longrightarrow}   +\infty.$$

To conclude, such $C$ does not exist.
\end{proof}

\subsection{Stability in weak convergence of measures in 1 dimension}
In the meantime,  we still verify that the optimizers satisfy weak convergence with respect to the change of the initial distributions.
\begin{lemma}\label{lemma:weak convergence}
 In 1 dimension, if $\mu_l\rightharpoonup \mu$ weakly, then the densities of $\nu_l$ converges to the density of $\nu$ in $L^1(O)$, for the corresponding  optimizers of \eqref{primal}.
\end{lemma}

\begin{proof}
  Without loss of generality, we can let the underlying domain $O$ be $O=(-1,1)$.  From the proof of Theorem~\ref{thm:main}, the optimal solution to \eqref{primal} is given by \eqref{shape}, thus determined by the mass and the first moment. We have the following due to weak convergence: 
\begin{align*}
   \lim \limits_{l\to \infty} k_l= \lim \limits_{l\to \infty} \int_{-1}^1 1\mu_l(x) \ dx&=\int_{-1}^11\mu(x) \ dx=k, \text{ and}\\
  \lim \limits_{l\to \infty} \beta_l= \lim \limits_{l\to \infty} \int_{-1}^1 x\mu_l(x) \ dx&=\int_{-1}^1 x\mu(x) \ dx=\beta.
\end{align*}

Also, $\nu_l=\chi_{A_l}$, where $$A_l=\left(-1,-1+\frac{k_l}{2}-\frac{\beta_l}{2-k_l}\right)\bigcup \left(1-\frac{k_l}{2}-\frac{\beta_l}{2-k_l}, 1\right),$$ and $\nu=\chi_{A}$, where $A=\left(-1,-1+\frac{k}{2}-\frac{\beta}{2-k}\right)\bigcup \left(1-\frac{k}{2}-\frac{\beta}{2-k}, 1\right)$. With the limit above, the densities of $\nu_l$ converges to the density of $\nu$ in $L^1(O)$.
\end{proof}

\section{Appendix}
\subsection{Proof of unique stationary point on $(-1,1)$ in Step 1 of Theorem~\ref{thm:potentials}}\label{sec:appendix}
 We prove here that there is only one solution for $[U^{\nu^*}]'(s)-[U^{\mu}]'(s)=0$  on $(-1,1)$, namely, $s=\dfrac{2\beta(1-k)}{k(2-k)}$.

\begin{proof}
This is proved through computations. 
As we assume $-1<a<b<1$, $0<k< 2$ and \begin{align*}
        \int_{-1}^{-1+k} x \ dx<&\beta< \int_{1-k}^1 x \ dx,\\
    \dfrac{k^2}{2}-k < &\beta< k-\dfrac{k^2}{2}.
    \end{align*}  
There are nine possibilities for the position of stationary point $s\in(-1,1)$.
    \begin{enumerate}[label=\textbf{Case  \arabic*.},start=1]
    
\item $1-\dfrac{k}{2}-\dfrac{\beta}{2-k}\leq s\leq1$ and $s\geq \dfrac{\beta}{k}+\dfrac{k}{2}$.

Then $[U^{\nu^*}(s)]'-[U^{\mu}(s)]'=0$ implies $s=1$. But we assume $s\in (-1,1)$ only, so we can ignore this case.

\item $1-\dfrac{k}{2}-\dfrac{\beta}{2-k}\leq s\leq 1$ and $ \dfrac{\beta}{k}-\dfrac{k}{2}\leq s\leq \dfrac{\beta}{k}+\dfrac{k}{2}$.

Then $[U^{\nu^*}(s)]'-[U^{\mu}(s)]'=0$ implies $\beta=k-\dfrac{k^2}{2}$ 
But $\beta<k-\dfrac{k^2}{2}$, so this is impossible.

\item $1-\dfrac{k}{2}-\dfrac{\beta}{2-k}\leq s\leq1$ and $s\leq \dfrac{\beta}{k}-\dfrac{k}{2}$. Then
  \begin{align*}
   0\geq  1-\dfrac{k}{2}-\dfrac{\beta}{2-k}-\left[\dfrac{\beta}{k}-\dfrac{k}{2} \right]&=1-\dfrac{2\beta}{k(2-k)}\\
    &> 1-\dfrac{2\left(k-\dfrac{k^2}{2}\right)}{k(2-k)}=0
\end{align*}

Contradiction arises. So this case can be ignored.

\item $-1+\dfrac{k}{2}-\dfrac{\beta}{2-k}\leq s\leq 1-\dfrac{k}{2}-\dfrac{\beta}{2-k}$ and $s\geq \dfrac{\beta}{k}+\dfrac{k}{2}$.

Then $[U^{\nu^*}(s)]'-[U^{\mu}(s)]'=0$ implies $\beta=\dfrac{k^2}{2}-k$. But $\beta>\dfrac{k^2}{2}-k$, so this is impossible.

\item $-1+\dfrac{k}{2}-\dfrac{\beta}{2-k}\leq s\leq 1-\dfrac{k}{2}-\dfrac{\beta}{2-k}$ and $\dfrac{\beta}{k}-\dfrac{k}{2} \leq s \leq \dfrac{\beta}{k}+\dfrac{k}{2}$.

Then $[U^{\nu^*}(s)]'-[U^{\mu}(s)]'=0$ implies $ s=\dfrac{2\beta(1-k)}{k(2-k)}$.

\item $-1+\dfrac{k}{2}-\dfrac{\beta}{2-k}\leq s\leq 1-\dfrac{k}{2}-\dfrac{\beta}{2-k}$ and $s\leq \dfrac{\beta}{k}-\dfrac{k}{2}$.

Then $[U^{\nu^*}(s)]'-[U^{\mu}(s)]'=0$ implies $\beta=k-\dfrac{k^2}{2}$. But $\beta<k-\dfrac{k^2}{2}$, so this is impossible.

\item $-1\leq s\leq -1+\dfrac{k}{2}-\dfrac{\beta}{2-k}$ and $s\geq \dfrac{\beta}{k}+\dfrac{k}{2}$. Then \begin{align*}
  0\leq -1+\dfrac{k}{2}-\dfrac{\beta}{2-k}-\left[\dfrac{\beta}{k}+\dfrac{k}{2} \right]&=-1-\dfrac{2\beta}{k(2-k)}\\
    &< -1+\dfrac{2\left(k-\dfrac{k^2}{2}\right)}{k(2-k)}=0
\end{align*}
Contradiction arises. So this case can be ignored.

\item $-1\leq s \leq -1+\dfrac{k}{2}-\dfrac{\beta}{2-k}$ and $ \dfrac{\beta}{k}-\dfrac{k}{2}\leq s\leq \dfrac{\beta}{k}+\dfrac{k}{2}$.

Then $[U^{\nu^*}(s)]'-[U^{\mu}(s)]'=0$ implies $\beta=\dfrac{k^2}{2}-k$ . But $\beta>\dfrac{k^2}{2}-k$, so this is impossible.

\item $-1\leq s\leq -1+\dfrac{k}{2}-\dfrac{\beta}{2-k}$ and $s\leq \dfrac{\beta}{k}-\dfrac{k}{2}$. 

Then $[U^{\nu^*}(s)]'-[U^{\mu}(s)]'=0$ implies $s=-1$.
But we assume $s\in (-1,1)$ only, so this is impossible.
\end{enumerate}

From the above, we can conclude that there is only one solution for $[U^{\nu^*}]'(s)-[U^{\mu}]'(s)=0$  on $(-1,1)$, namely, $s=\dfrac{2\beta(1-k)}{k(2-k)}$.
\end{proof}

\newpage

\medskip

\bibliographystyle{plain}

\bibliography{bib}

\end{document}